\documentclass[11pt]{amsart}
\expandafter\let\csname ver@amsthm.sty\endcsname\relax
\let\theoremstyle\relax

\usepackage[colorlinks=true,linkcolor=blue]{hyperref}
\usepackage{amsmath,amsthm,amssymb}
\usepackage{tikz}
\usepackage[capitalize,nameinlink]{cleveref}
\crefname{claim}{Claim}{Claims}

\title{Hardness Results for the Subpower Membership Problem}
\author[J. Shriner]{Jeff Shriner}
\address[Jeff Shriner]{Department of Mathematics, University of Colorado Boulder, CO 80309-0395, USA}
\email{jeffrey.shriner@colorado.edu}

\thanks{This material is based upon work supported by the National Science Foundation under Grant No. DMS 1500254.}

\newcommand{\dproblem}[3]{
\begin{equation*}
\begin{tabular}{  l c  p{0.70\textwidth}  }
& \multicolumn{2}{l}{\textbf{\uppercase{#1}}} \\
& Instance: & #2  \\ 
& Question: & #3 \\
\end{tabular}
\end{equation*}
}

\newcommand{\B}{\mathbb{B}}
\newcommand{\A}{\mathbb{A}}
\newcommand{\T}{\mathfrak{T}}
\newcommand{\M}{\mathcal{M}}
\newcommand{\Id}{\mathbb{I}^\mathrm{d}}
\newcommand{\tod}{\to^\mathrm{d}}
\newcommand{\F}{\mathcal{F}}
\renewcommand{\H}{\mathcal{H}}
\newcommand{\V}{\mathcal{V}}
\newcommand{\W}{\mathcal{W}}

\newcommand{\oE}{\overline{\Sigma}}
\renewcommand{\phi}{\varphi}

\newcommand{\SMP}{\textnormal{SMP}}
\newcommand{\CD}{\textnormal{CD}}
\newcommand{\CP}{\textnormal{CP}}

\theoremstyle{plain}
\newtheorem{theorem}{Theorem}[section]
\newtheorem{corollary}[theorem]{Corollary}
\newtheorem{lemma}[theorem]{Lemma}
\newtheorem{claim}[theorem]{Claim}

\theoremstyle{definition}
\newtheorem{question}[theorem]{Question}
\newtheorem{ack}[theorem]{Acknowledgements}

\begin{document}

\begin{abstract}
The main result of this paper shows that if $\M$ is a consistent strong linear Maltsev condition which does not imply the existence of a cube term, then for any finite algebra $\A$ there exists a new finite algebra $\A_\M$ which satisfies the Maltsev condition $\M$, and whose subpower membership problem is at least as hard as the subpower membership problem for $\A$. We characterize consistent strong linear Maltsev conditions which do not imply the existence of a cube term, and show that there are finite algebras in varieties that are congruence distributive and congruence $k$-permutable ($k \geq 3$) whose subpower membership problem is EXPTIME-complete.
\end{abstract}

\maketitle

\section{Introduction}
We are interested in analyzing the time complexity of the subpower membership problem, which is a combinatorial decision problem asking if a given element of the direct power of an algebra is in the subalgebra generated by a given set of generators. More precisely, for a fixed finite algebra $\A$ in a language with finitely many operation symbols, the subpower membership problem for $\A$ is the decision problem
\dproblem{SMP($\A$)}{A positive integer $m$ and $m$-tuples $a_1, \dots, a_n,b \in \A^m$.}{Is $b$ in the subalgebra $\langle a_1, \dots, a_n \rangle$ of $\A^m$ generated by $a_1,\dots,a_n$?}
The time complexity of this problem is measured with respect to the input size $(n+1)m$. A naive algorithm, which computes the full subalgebra $\langle a_1 ,\dots, a_n\rangle$ and checks if $b$ is a member, shows the problem can be answered in at most exponential time. M. Kozik \cite{Koz} provided an example which shows this problem can be EXPTIME-complete. We know many cases in which this problem can be answered in polynomial time, such as finite groups (and multilinear expansions of groups) \cite{FHL,Wil}, finite algebras with a near-unanimity term \cite{BP}, expansions of finite nilpotent Maltsev algebras of prime power size \cite{May}, and finite algebras with a cube term from residually small varieties \cite{BMS}. There has also been much work in classifying this problem for finite semigroups \cite{BKMS,St1,St2}.  

However, there are still large unanswered questions regarding the time complexity of the subpower membership problem, even for finite algebras in congruence modular varieties. The time complexity remains unknown for the general case of finite Maltsev algebras, expansions of groups, and near-rings. 

In this paper, we provide hardness results for the subpower membership problem. A finite set of linear identities is called a strong linear Maltsev condition. Examples of strong linear Maltsev conditions are $k$-permutability for fixed $k \geq 1$ and the existence of an $m$-cube term for an algebra for fixed $m \geq 2$. Given a consistent strong linear Maltsev condition $\M$ which does not imply the existence of a cube term and given any finite algebra $\A$, we will construct a new finite algebra $\A_\M$ which satisfies the Maltsev condition $\M$, and whose subpower membership problem is at least as hard as the subpower membership problem for $\A$. We will characterize consistent strong linear Maltsev conditions which do not imply the existence of a cube term, and apply these results to Kozik's algebra \cite{Koz} which has EXPTIME-complete SMP to construct finite algebras in varieties that are congruence distributive and congruence $k$-permutable ($k \geq 3$) whose subpower membership problem is EXPTIME-complete.

\section{Preliminaries}
For a fixed algebraic language, a term in that language is called \textit{linear} if it contains at most one operation symbol, and an identity $s \approx t$ is called \textit{linear} if both $s$ and $t$ are linear terms. If $\Sigma \cup \{\phi\}$ is a set of linear identites, then $\phi$ is a \textit{consequence} of $\Sigma$, written $\Sigma \models \phi$, if every model of $\Sigma$ is a model of $\phi$. David Kelly's Completeness Theorem \cite{Kel,KKS} characterizes the $\models$ relation using a simple proof system for linear identities, which we will now describe.

If $\Sigma$ is a set of linear identities over the variable set $X$, the \textit{weak closure of $\Sigma$ in the variables $X$} is the smallest set $\oE$ of linear identities containing $\Sigma$ for which
	\begin{enumerate}
	\item $u \approx u \in \oE$ for all linear terms $u$ with variables from $X$.
	\item If $u \approx v \in \oE$, then $v \approx u \in \oE$.
	\item If $u \approx v \in \oE$ and $v \approx w \in \oE$, then $u \approx w \in \oE$.
	\item If $u \approx v \in \oE$ and $\gamma : X \to X$ is a function, then $u[\gamma] \approx v[\gamma] \in \oE$, where $u[\gamma]$ denotes the linear term obtained from $u$ by replacing each variable $x \in X$ with $\gamma(x) \in X$. 
	\end{enumerate}	
We write $\Sigma \vdash_{X} \phi$ if $\phi \in \oE$. Kelly's Completeness Theorem states that $\Sigma \models \phi$ if and only if $\Sigma \vdash_{X} \phi$ or $\Sigma \vdash_{X} x\approx y$ (for $x \neq y$), provided that $X$ is \textit{large enough} for $\Sigma \cup \{\phi\}$; that is, 
	\begin{itemize}
	\item $X$ contains at least $2$ variables,
	\item $|X| \geq \mathrm{arity}(f)$ for any operation symbol $f$ occurring in $\Sigma$, and
	\item $|X|$ is at least as large as the number of distinct variables occurring in any identity in $\Sigma \cup \{\phi\}$. 
	\end{itemize}

If $X$ and $Y$ are variable sets both large enough for $\Sigma \cup \{\phi\}$, then Kelly's Completeness Theorem implies that $\Sigma \vdash_{X} \phi$ if and only if $\Sigma \vdash_{Y} \phi$. Thus, if $\Sigma \vdash_{X} \phi$ for some $X$ which is large enough for $\Sigma \cup \{\phi\}$, we simply write $\Sigma \vdash \phi$ and say \textit{$\Sigma$ entails $\phi$}. Accordingly, we will refer to properties (1) through (4) above as \textit{entailment properties}. 

We say $\Sigma$ is \textit{inconsistent} if $\Sigma$ entails $x \approx y$ for distinct variables $x$ and $y$. Using Kelly's Completeness Theorem, we see $\Sigma$ is inconsistent if and only if the only models of $\Sigma$ are the trivial algebras. If $\Sigma$ is not inconsistent (or equivalently, has a non-trivial model), we say $\Sigma$ is \textit{consistent}. 

A \textit{strong Maltsev condition} is a pair $\M = (\H , \Sigma)$, where $\H$ is a finite set of operation symbols, and $\Sigma$ is a finite set of identites involving terms constructed from members of $\H$. A strong Maltsev condition is \textit{linear} if all of the identities in $\Sigma$ are linear, and a strong linear Maltsev condition is \textit{consistent} if $\Sigma$ is consistent. 

Finally, we define a particular strong linear Maltsev condition, $(\{c\}, \Sigma)$. The identities of $\Sigma$ are given by the rows of 
\[
c\left(
\begin{bmatrix}
\\
x_1\\
\\
\end{bmatrix},
\dots,
\begin{bmatrix}
\\
x_n\\
\\
\end{bmatrix}
\right)
\approx
\begin{bmatrix}
y\\
\vdots\\
y 
\end{bmatrix},
\]
where $x_1,\dots,x_n \in \{x,y\}^m \setminus (y,\dots,y)$ for some positive integer $m \geq 2$. We refer to $\Sigma$ as \textit{a set of cube identities for $c$}. If an algebra $\A$ has a term $c$ which satisfies a set of $m$ cube identities for $c$, then we say that $c$ is an \textit{$m$-cube term} for $\A$.

\section{The SMP and Strong Linear Maltsev Conditions}
We are now ready to state and prove the main theorem of this paper.

\begin{theorem}\label{cube}
Let $\M = (\H, \Sigma)$ be a consistent strong linear Maltsev condition such that $\Sigma$ does not entail cube identities for any $h \in \H$. 
	\begin{enumerate}
	\item[{\rm(i)}] For any finite algebra $\A$, there exists a finite algebra $\A_{\M}$ such that the language of $\A_\M$ contains $\H$, $\A_{\M} \models \Sigma$, and $\SMP(\A_{\M})$ is at least as hard as $\SMP(\A)$.
	\item[{\rm(ii)}] In particular, there exists a finite algebra $\B_\M$ such that the language of $\B_\M$ contains $\H$, $\B_\M \models \Sigma$, and $\SMP(\B_\M)$ is {\rm EXPTIME}-complete.
	\end{enumerate}
\end{theorem}
\begin{proof}
(i) Let $\A = \langle A; \F \rangle$ be a finite algebra. We assume the languages $\F$ and $\H$ are disjoint, and now define $\A_{\M}$. We define the universe of $\A_{\M}$ to be the set $A \cup \{0\}$, where the element $0$ is distinct from all elements of $A$. The language of $\A_{\M}$ will be $\F \cup \H$, which we must interpret in $\A_{\M}$. For $k$-ary $f \in \F$ and $(a_1,\dots, a_k) \in (A\cup \{0\})^k$, we define $f^{\A_{\M}}(a_1,\dots,a_k)= f^{\A}(a_1,\dots,a_k)$ if $a_i \neq 0$ for all $1 \leq i \leq k$, and $f^{\A_{\M}}(a_1,\dots,a_k) =0$ otherwise. Thus, $0$ is an absorbing element with respect to the operations $f^{\A_\M}$ for $f\in \F$.

To interpret the operation symbols of $\H$ in $\A_{\M}$, we need to introduce some terminology and notation. Let $X$ be a variable set which is large enough for $\Sigma$. Since $\M$ is consistent, we know $\Sigma$ does not entail $x \approx y$ for distinct $x,y \in X$. For any positive integer $k$, we will say that $(x_1,\dots,x_k) \in X^k$ and $(a_1,\dots,a_k) \in A^k$ have the \textit{same equality pattern} if, for all $1 \leq i,j \leq k$, the tuples have the property that $x_i = x_j$  if and only if $a_i = a_j$. For $\overline{a} \in A^k$, define 
\[
P_{\overline{a}} = \{\overline{x} \in X^k \mid \overline{x} \text{ and } \overline{a} \text{ have the same equality pattern} \}.
\]

We are now ready to describe the interpretation of the symbols in $\H$. For $k$-ary $h \in \H$ and $\overline{a} = (a_1,\dots,a_k) \in (A \cup \{0\})^k$, define
\[
h^{\A_\M}(a_1,\dots, a_k) = 
\begin{cases}
a_i &\text{if there exist } (x_1,\dots, x_k) \in P_{\overline{a}} \text{ and } 1\leq i \leq  k \\
    &\text{such that } \Sigma \vdash  h(x_1,\dots,x_k)  \approx x_i  \\
0 &\text{otherwise. }
\end{cases}
\]
We first show $h^{\A_\M}$ is well-defined for each $h \in \H$. We must show that if $(y_1, \dots, y_k), (z_1,\dots, z_k) \in P_{\overline{a}}$, and 
\[
\Sigma \vdash  h(y_1,\dots, y_k) \approx y_r \text{ and } \Sigma \vdash h(z_1,\dots,z_k) \approx z_q,  \text{ $1 \leq r,q \leq k$},
\] 
then $a_r = a_q$. To see this is the case, note that $(y_1, \dots, y_k), (z_1,\dots, z_k) \in P_{\overline{a}}$ implies that $y_i = y_j$ if and only if $z_i = z_j$ for all $1 \leq i,j \leq k$. Thus, the map $\gamma \colon \{y_1,\dots,y_k\} \to \{z_1,\dots,z_k\}$, $y_i \mapsto z_i$, is well-defined, and by entailment property (4) we have that $\Sigma \vdash h(z_1,\dots,z_k) \approx z_r$. Thus, by entailment properties (2) and (3), $\Sigma \vdash z_r \approx z_q$. Since $\Sigma$ is consistent it must be that $z_r = z_q$, hence $(z_1,\dots, z_k) \in P_{\overline{a}}$ implies that $a_r = a_q$. This completes the definition of $\A_{\M} = \langle A \cup \{0\} ; \F \cup \H \rangle$.

Next we show $\A_\M \models \Sigma$. In order to show this, we will first discuss how to evaluate, in $\A_\M$, an arbitrary linear term in the language $\H$. In the following, we use ``$=$" to denote equality of terms. Let $w(y_1,\dots,y_\ell)$ be any linear term with distinct variables $y_1,\dots, y_\ell$, where $\ell \leq |X|$ and $w$ need not depend on all variables. Since $w$ is a linear term, $w(y_1,\dots, y_\ell) = h(y_{t(1)},\dots, y_{t(k)})$ for some $k$-ary $h \in \H$ and some map $t \colon \{1,\dots, k\} \to \{1,\dots, \ell\}$.

\begin{claim}\label{eval}
Let $w(y_1,\dots,y_\ell) = h(y_{t(1)},\dots, y_{t(k)})$ be a linear term in the language $\H$ $(\ell \leq |X|)$. If $\overline{a} =(a_1,\dots,a_\ell)\in (A \cup \{0\})^\ell$, then 
\[
w^{\A_{\M}}(a_1,\dots, a_\ell) =
\begin{cases}
a_{j} &\text{if there exist } (x_1,\dots, x_\ell) \in P_{\overline{a}} \text{ and } 1\leq j \leq  \ell \\
    	&\text{such that } \Sigma \vdash w(x_1,\dots,x_\ell) \approx x_j   \\
0 &\text{otherwise. } 
\end{cases}
\]
\end{claim}
\begin{proof} 
\renewcommand{\qedsymbol}{$\blacksquare$}
Let $\overline{a} =(a_1,\dots,a_\ell)\in (A \cup \{0\})^\ell$, and let $\overline{a}_t = (a_{t(1)},\dots,a_{t(k)}) \in (A \cup \{0\})^k $. We first show that the following two conditions on $\overline{a}$ and $\overline{a}_t$ are equivalent:
	\begin{enumerate}
	\item[(a)] There exist $(x_1, \dots ,x_\ell) \in P_{\overline{a}}$ and $1 \leq j \leq \ell$ such that 
	\[
	\Sigma \vdash  w(x_1,\dots, x_\ell) \approx x_j.
	\]
	\item[(b)] There exist $(x_{t(1)}, \dots , x_{t(k)}) \in P_{\overline{a}_t}$ and $1 \leq i \leq k$ such that 
	\[
	\Sigma \vdash  h(x_{t(1)},\dots,x_{t(k)}) \approx x_{t(i)}.
	\]
	\end{enumerate}
(a) $\Rightarrow$ (b) If $(x_1, \dots ,x_\ell) \in P_{\overline{a}}$, then $(x_{t(1)}, \dots , x_{t(k)}) \in P_{\overline{a}_t}$. Further,
\begin{align*}
\Sigma \vdash x_j &\approx w(x_1,\dots,x_\ell) \text{ by assumption and entailment property (2)} \\
	 &=h(x_{t(1)}, \dots , x_{t(k)}).
\end{align*}
Thus, $\Sigma \vdash h(x_{t(1)}, \dots , x_{t(k)}) \approx x_j$ by entailment property (2). Since $\Sigma$ is consistent, $j = t(i)$ for some $1 \leq i \leq k$. 

\noindent (b) $\Rightarrow$ (a) Let $\equiv$ denote the equivalence relation on the set $\{1,\dots, \ell\}$ defined by $r \equiv q$ if and only if $a_r = a_q$, and denote the $\equiv$-class of $r \in \{1,\dots,\ell\}$ by $[r]$. Let $T$ be the set of $\equiv$-classes. Since $|T| \leq \ell \leq |X|$ and 
\[
t(q) \equiv t(r) \iff a_{t(q)} = a_{t(r)} \iff x_{t(q)} = x_{t(r)},
\] 
there is a well-defined and one-to-one map $\psi \colon T \to X$ such that $\psi([t(r)]) = x_{t(r)}$ for all numbers of the form $t(r)$ ($1 \leq r \leq k$). For each $s \in \{1,\dots, \ell\}$, define $z_s := \psi([s])$. Then $(z_1,\dots, z_\ell) \in P_{\overline{a}}$, and $x_{t(r)} = \psi([t(r)]) = z_{t(r)}$ for all $1 \leq r \leq k$. We compute
\begin{align*}
\Sigma \vdash x_{t(i)} &\approx  h(x_{t(1)}, \dots, x_{t(k)}) \text{ by assumption and entailment property (2)} \\
			&=h(z_{t(1)}, \dots, z_{t(k)}) \\
			&= w(z_1,\dots, z_\ell) .
\end{align*}
By entailment property (2), we have $\Sigma \vdash w(z_1,\dots, z_\ell) \approx x_{t(i)}$. Since $x_{t(i)} = z_{t(i)}$, $\Sigma \vdash w(z_1,\dots, z_\ell) \approx z_{t(i)}$, where $1 \leq t(i) \leq \ell$.  This completes the argument that (a) and (b) are equivalent. 

Now we compute
\begin{align*}
w^{\A_{\M}}(a_1,\dots, a_\ell) &= h^{\A_{\M}}(a_{t(1)},\dots, a_{t(k)}) \\
			&=\begin{cases}
			a_{t(i)} &\text{if } \exists (x_{t(1)},\dots, x_{t(k)}) \in P_{\overline{a}_t} \text{ and } 1\leq i \leq  k  \\
    				&\text{such that } \Sigma \vdash  h(x_{t(1)},\dots,x_{t(k)}) \approx x_{t(i)}   \\
			0 &\text{otherwise }
			\end{cases}\\
			&=\begin{cases}
			a_{j} &\text{if } \exists (x_1,\dots, x_\ell) \in P_{\overline{a}} \text{ and } 1\leq j \leq  \ell \\
    				&\text{such that } \Sigma \vdash w(x_1,\dots,x_\ell) \approx x_j   \\
			0 &\text{otherwise, } 
			\end{cases}
\end{align*}
where the last equality follows from the equivalence of (a) and (b). 
\end{proof}

We now finish the argument that $\A_\M \models  \Sigma$. Let $u \approx v \in \Sigma$, where $u$ and $v$ are linear terms in the language $\H$. Since $X$ is large enough for $\Sigma$, we may write $u(y_1,\dots,y_\ell) \approx v(y_1,\dots,y_\ell)$, where $y_1,\dots,y_\ell$ are distinct variables ($\ell \leq |X|$) and $u$ and $v$ need not depend on every variable.  If $\overline{a} =(a_1,\dots,a_\ell) \in (A\cup \{0\})^\ell$, entailment properties (2), (3), and (4) imply that if $(x_1,\dots,x_\ell) \in P_{\overline{a}}$ and $1 \leq j \leq \ell$, then $\Sigma \vdash u(x_1,\dots,x_\ell) \approx x_j$ if and only if $\Sigma \vdash v(x_1,\dots, x_\ell) \approx x_j$. Thus by \cref{eval}, $u^{\A_\M}(a_1,\dots, a_\ell) = v^{\A_\M}(a_1,\dots, a_\ell)$, so $\A_\M \models u \approx v$. 

To show that $\SMP(\A_{\M})$ is at least as hard as $\SMP(\A)$, we must reduce an instance of $\SMP(\A)$ to an instance of $\SMP(\A_{\M})$ in polynomial time such that the $\SMP(\A)$ instance has a `yes' answer if and only if the $\SMP(\A_{\M})$ instance has a `yes' answer. 

Fix an instance $a_1,\dots, a_n,b \in A^m$ of $\SMP(\A)$. This is also an instance of $\SMP(\A_{\M})$, and we will use this same instance in our reduction. Since we have not changed the instance, this reduction can be done in constant time.

The main goal now is to show that $b$ is in the subalgebra of $\A^m$ generated by $a_1,\dots, a_n$ if and only if $b$ is in the subalgebra of $\A_\M^m$ generated by $a_1,\dots, a_n$. To distinguish generated subalgebras of $\A^m$ and generated subalgebras of $\A_{\M}^m$, we will denote the subalgebra $\langle a_1,\dots, a_n\rangle$ of $\B \in \{\A^m,\A_{\M}^m\}$ by $\langle a_1,\dots, a_n\rangle_\B$. 

Suppose first that the $\SMP(\A)$ instance has a `yes' answer. That is, $b \in \langle a_1, \dots, a_n \rangle_{\A^m}$. Let $p(x_1,\dots, x_n)$ be a term in the language $\F$ such that $p^{\A^m}(a_1,\dots,a_n) =b$. Then $p(x_1,\dots,x_n)$ is also a term in the language $\F \cup \H$ and $p^{\A^m}(a_1,\dots,a_n) = p^{\A_\M^m}(a_1,\dots,a_n)$, so $b \in \langle a_1,\dots,a_n\rangle_{\A_\M^m}$. Thus the $\SMP(\A_{\M})$ instance also has a `yes' answer. 

For the converse direction, we will show that if $d_1,\dots, d_n \in (A \cup \{0\})^m$, $e \in A^m$, and $e \in \langle d_1, \dots, d_n \rangle_{\A_\M^m}$, then there exists a term $p(x_1,\dots,x_n)$ in the language $\F$ such that $p^{\A_\M^m}(d_1,\dots, d_n) =e$. Since $a_1,\dots, a_n, b \in A^m \subseteq (A \cup \{0\})^m$, this will show that $b \in  \langle a_1, \dots, a_n \rangle_{\A_\M^m}$ implies $b \in  \langle a_1, \dots, a_n \rangle_{\A^m}$. 

It will be useful to visualize a term by its term tree. We use the convention that the leaves are labeled by variables, and every node which is not a leaf is labeled by a single operation symbol. An example in the language $f$ (unary), $g$ (unary), and $t$ (binary) is given in \cref{fig1}.
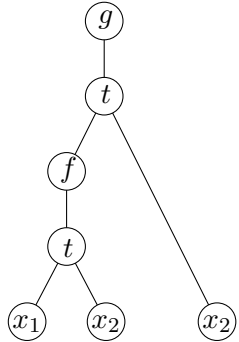
\begin{figure}[hbt]
	\begin{tikzpicture} 
        \draw[] (0,0) -- (0,-1);
        \draw[] (0,-1) -- (.525,-2);
        \draw[] (0,-1) -- (-.525,-2);
        \draw (0.5,1) -- (0,0);
        \draw (0.5,1) -- (2,-2);
        \draw (0.5,2) -- (0.5,1);
        \draw[fill=white] (0.5,2) circle (.25cm);
        \draw[fill=white] (0,0) circle (.25cm);
        \draw[fill=white] (0,-1) circle (.25cm);
        \draw[fill=white] (.525,-2) circle (.25cm);
        \draw[fill=white] (2,-2) circle (.25cm);
        \draw[fill=white] (-.525,-2) circle (.25cm);
        \draw[fill=white] (0.5,1) circle (.25cm);
        \node at (.025,.02) {$f$};
        \node at (.025,-1.02) {$t$};
        \node at (.525,-2.025) {$x_2$};
        \node at (-.525,-2.025) {$x_1$};
        \node at (2, -2.025) {$x_2$};
        \node at (0.525,1.025) {$t$};
        \node at (0.525, 2.025) {$g$};
	\end{tikzpicture}
\caption{The term tree $\T_r$ for the term $r(x_1,x_2) = g(t(f(t(x_1 , x_2)) , x_2))$.}\label{fig1}
\end{figure}
We refer to the vertex of maximum height in this tree as the \textit{root}, and denote the term tree of a term $r$ by $\T_r$. 

\begin{claim}\label{elim}
If $d_1,\dots, d_n \in (A \cup \{0\})^m$, $e \in A^m$, and $e \in \langle d_1, \dots, d_n \rangle_{\A_\M^m}$, then there exists a term $p(x_1,\dots,x_n)$ in the language $\F$ such that
\[
p^{\A_\M^m}(d_1,\dots,d_n) =e.
\]
\end{claim}
\begin{proof}
\renewcommand{\qedsymbol}{$\blacksquare$}
Let $d_1,\dots, d_n \in (A \cup \{0\})^m$ and $e \in A^m$. Let $p(x_1,\dots,x_n)$ be a term in the language $\F \cup \H$ such that $p^{\A_\M^m}(d_1,\dots,d_n)=e$. We assume $p(x_1,\dots, x_n)$ was chosen so that $\T_p$ has the minimum number of vertices with labels from $\H$ and satisfies $p^{\A_\M^m}(d_1,\dots, d_n)=e$.  If $\T_p$ has no label from $\H$, then the claim is proven. So we assume $\T_p$ has at least one vertex with label from $\H$. We will analyze the term $p(x_1,\dots, x_n)$ in parallel with the evaluation of $p$ at $d_1,\dots,d_n$, as illustrated in \cref{fig2}.

Choose a maximal vertex with respect to height with label from $\H$, and say the label is $h \in \H$. Call this vertex $\nu$. The subtree of $\T_p$ whose root is $\nu$ corresponds to a subterm $q$ of $p$. If $h$ is $k$-ary, the vertex $\nu$ has $k$ edges corresponding to $k$ subterms of $q$, which we will denote as $s_1,\dots, s_k$. For $1\leq i \leq k$, we define $c_i:=s^{\A_\M^m}_i(d_1,\dots,d_n)$. Set $z :=h^{\A_\M^m}(c_1,\dots, c_k)$. For $1 \leq j \leq m$, we write $z|_j$ to denote the $j^{th}$ coordinate of the $m$-tuple $z$. 
\begin{figure}[hbt]
	\begin{tikzpicture}
	\node at (0,2.5) {$\T_p$};
	\draw[] (0,2) circle (.2cm);
	\draw (-.15,1.9) -- (-2.5,-2.5);
	\draw (.15,1.9) -- (2.5,-2.5);

	\node at (0,.4) {$\T_q$};
	\draw (0,0) circle (.2cm);
	\node at (0,0) {$h$};
	\draw (0,-.2) -- (-.35,-.35);
	\draw (0,-.2) -- (.35,-.35);
	\draw[] (-.5,-.5) circle (.2cm);
	\draw[] (.5,-.5) circle (.2cm);
	\node at (-.7,-.1) {$\T_{s_1}$};
	\node at (.7,-.1) {$\T_{s_k}$};
	\draw[dotted] (-.3,-.5) -- (.3,-.5);
	\draw (-.5,-.7) -- (-1.5,-2.5);
	\draw (.5,-.7) -- (1.5,-2.5);
	\draw[dotted] (-.75,-1.5) -- (.75,-1.5);
	\draw (-.2,.1) -- (-.4,.5);
	\draw (-.4,.5) -- (0,1);
	\draw (0,1) -- (0,1.3);
	\draw (0,1.3) -- (-.15,1.5);
	\draw (-.15,1.5) -- (0,1.8);

	\node at (6,2.5) {$p^{\A_\M^m}(d_1,\dots,d_n)$};
	\node at (6,2) {$e$};
	\draw[] (6,2) circle (.2cm);
	\draw (5.85,1.9) -- (3.5,-2.5);
	\draw (6.15,1.9) -- (8.5,-2.5);

	\draw (6,0) circle (.2cm);
	\node at (6,0) {$z$};
	\draw (6,-.2) -- (5.65,-.35);
	\draw (6,-.2) -- (6.35,-.35);
	\draw[] (5.5,-.5) circle (.2cm);
	\draw[] (6.5,-.5) circle (.2cm);
	\node at (5.5,-.5) {$c_1$};
	\node at (6.5,-.5) {$c_k$};
	\draw[dotted] (5.7,-.5) -- (6.3,-.5);
	\draw (5.5,-.7) -- (4.5,-2.5);
	\draw (6.5,-.7) -- (7.5,-2.5);
	\draw[dotted] (5.25,-1.5) -- (6.75,-1.5);
	\draw (5.8,.1) -- (5.6,.5);
	\draw (5.6,.5) -- (6,1);
	\draw (6,1) -- (6,1.3);
	\draw (6,1.3) -- (5.85,1.5);
	\draw (5.85,1.5) -- (6,1.8);
	\end{tikzpicture}
\caption{The term tree $\T_p$ for the term $p(x_1,\dots,x_n)$ in the language $\F \cup \H$ (left), and the evaluation of $p$ at $(d_1,\dots, d_n)$ (right). }\label{fig2}
\end{figure}
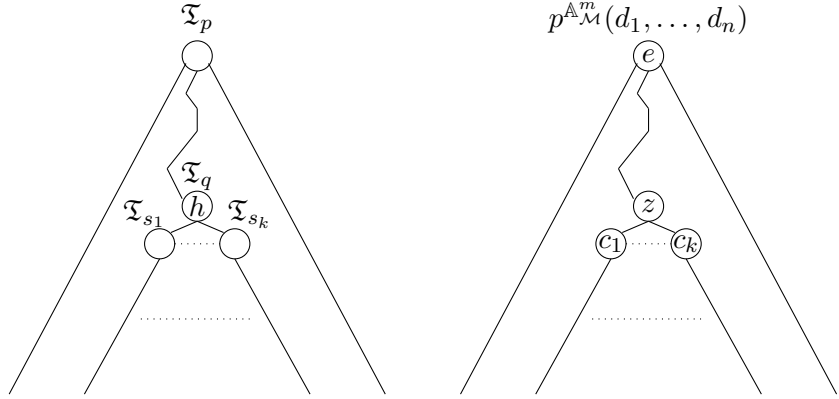
Since $\nu$ is a maximal vertex in $\T_p$ with label from $\H$, the term $p$ has the form $t(\dots,q(x_1,\dots,x_n),\dots)$, where $t$ is a term in the language $\F$. Since $0$ is an absorbing element with respect to the operations $f^{\A_\M}$ for $f \in \F$, and $e|_j \neq 0$ for all $1 \leq j \leq m$, we must have that 
\[
h^{\A_\M^m}(c_1,\dots, c_k)|_j =z|_j= q^{\A_\M^m}(d_1,\dots, d_n)|_j  \neq 0
\]
for all $1\leq j \leq m$. For $1 \leq j \leq m$, define
\[
B_j = \{i\in\{1,\dots, k\} \mid h^{\A_\M^m}(c_1,\dots, c_k)|_j = c_i|_j\}.
\]

If $\bigcap_{j=1}^m B_j \neq \emptyset$, we choose $i \in \bigcap_{j=1}^m B_j $ and form a new term $p'$ by replacing the subterm $q(x_1,\dots,x_n) = h(s_1(x_1,\dots,x_n),\dots, s_k(x_1,\dots,x_n))$ of $p$ with $s_i(x_1,\dots,x_n)$. Since $i \in \bigcap_{j=1}^m B_j $, we have $q^{\A_\M^m}(d_1,\dots,d_n) = h^{\A_\M^m}(c_1,\dots,c_k) = c_i$, so
\begin{align*}
p^{\A_\M^m}(d_1,\dots,d_n) &= t^{\A_\M^m}(\dots, q^{\A_\M^m}(d_1,\dots,d_n),\dots) \\
		&= t^{\A_\M^m}(\dots, c_i, \dots) \\
		&= t^{\A_\M^m}(\dots, s_i^{\A_\M^m}(d_1,\dots, d_n),\dots) \\
		&= p'^{\A_\M^m}(d_1,\dots,d_n).
\end{align*}
Thus, $p'^{\A_\M^m}(d_1,\dots, d_n) = e$. Further, $\T_{p'}$ has fewer vertices with labels from $\H$ than $\T_p$, which contradicts the choice of the term $p$.

Thus, it must be that $\bigcap_{j=1}^m B_j = \emptyset$. For each $1\leq j\leq m$, let $\overline{c}|_j =(c_1, \dots, c_k)|_j$. Since $h^{\A_\M^m}(c_1, \dots, c_k)|_j\neq 0$, there exist $(x_1^j,\dots,x_k^j) \in P_{\overline{c}|_j}$ and $1 \leq \ell \leq k$ such that $\Sigma \vdash h(x_1^j,\dots,x_k^j) \approx x_\ell^j$. In particular, $\ell \in B_j$. Define $\gamma_j \colon \{x_1^j,\dots,x_k^j\} \to \{x,y\}$, for distinct variables $x,y \in X$, by 
\[
\gamma_j(x_i^j) =
\begin{cases}
y &\text{ if } i \in B_j  \\
x &\text{ otherwise}.
\end{cases}
\]
This map is well-defined since $x_u^j = x_w^j$ if and only if $c_u|_j =c_w|_j$, which implies ($u\in B_j \iff w \in B_j)$. Then computing $h[\gamma_j]$ for all $1 \leq j \leq m$ and using entailment property (4), we have that 
\[
\Sigma \vdash 
h\left(
\begin{matrix}
\gamma_1(x_1^1) & \gamma_1(x_2^1) & \dots & \gamma_1(x_k^1) \\
\gamma_2(x_1^2) & \gamma_2(x_2^2) & \dots & \gamma_2(x_k^2) \\
	\vdots	&  \vdots			     & \vdots &	\vdots		\\
\gamma_m(x_1^m) & \gamma_m(x_2^m) & \dots & \gamma_m(x_k^m) \\
\end{matrix}
\right)
\approx
\begin{pmatrix}
y \\
y \\
\vdots \\
y
\end{pmatrix}.
\]
Since $\bigcap_{j=1}^m B_j = \emptyset$, no column in the above matrix on the left hand side is the tuple $(y, \dots, y)$. Thus, $\Sigma$ entails cube identities for $h$. This is also a contradiction, so we must have that $p$ is a term in the language $\F$. 
\end{proof}

Thus, if $b \in \langle a_1, \dots, a_n \rangle_{\A_\M^m}$, then there is a term $p(x_1,\dots,x_n)$ in the language $\F$ such that $p^{\A_\M^m}(a_1,\dots, a_n)=b$. Since $a_1,\dots,a_n \in A^m$, we have $p^{\A_\M^m}(a_1,\dots, a_n)=p^{\A^m}(a_1,\dots, a_n)$, so $b \in \langle a_1,\dots,a_n\rangle_{\A^m}$. We have thus shown that $\SMP(\A_{\M})$ is at least as hard as $\SMP(\A)$, which completes the proof of statement (i) of the theorem.

(ii) Let $\B$ be the finite algebra of Kozik \cite{Koz} for which the problem $\SMP(\B)$ is EXPTIME-complete. Then by statement (i) of the theorem, there exists a finite algebra $\B_\M$ such that $\B_\M \models \Sigma$ and $\SMP(\B_\M)$ is at least as hard as $\SMP(\B)$. Since the subpower membership problem can always be answered in EXPTIME, it follows that $\SMP(\B_\M)$ is EXPTIME-complete.
\end{proof}

\section{Applications}
We discuss some consequences of \cref{cube}. We will prove a characterization of consistent strong linear Maltsev conditions which do not imply the existence of a cube term, similar to the results of Opr\v{s}al \cite{Opr} and Moore and McKenzie \cite{MM}. We will use this characterization along with \cref{cube} to show there exist examples of finite algebras in varieties that are congruence distributive and congruence $k$-permutable ($k \geq 3$) whose SMP is EXPTIME-complete. Before stating and proving the corollaries, we must first introduce some definitions and notation.

Let $\V$ and $\W$ be two varieties, and let $\{f_i\}_{i \in I}$ be the languange of $\V$. We say that \textit{$\V$ is interpretable in $\W$} if for every operation symbol $f_i$, there is a term $t_i$ (of the same arity) in the language of $\W$ such that for all $\A \in \W$, the algebra $\langle A; \{t_i^{\A}\}_{i \in I} \rangle$ is a member of $\V$. If $\V$ is interpretable in $\W$, we write $\V \leq \W$. For a strong Maltsev condition $\M = (\H,\Sigma)$, we denote the variety determined by $\Sigma$ by $\V_\M$. 

The \textit{dual algebra} of the $2$-element implication algebra $\mathbb{I} = \langle \{0,1\}; \to \rangle$ is the algebra $\mathbb{I}^\mathrm{d} = \langle \{0,1\}; \tod \rangle$, where the operation $\tod$ is binary and is obtained from the operation table of $\to$ by permuting $0$ and $1$. 
\begin{table}[hbt]\label{tod}
	\begin{tabular}{l | c  r}
	$\tod$ & 0 & 1 \\ \hline
	0 & 0 & 1 \\
	1 & 0 & 0 
	\end{tabular}
\caption{The operation table for $\tod$.}
\end{table}
\begin{corollary}\label{int}
If $\M =(\H,\Sigma)$ is a strong linear Maltsev condition, then the following are equivalent:
	\begin{enumerate}
	\item[{\rm (i)}] $\M$ is consistent and $\Sigma$ does not entail the existence of cube identities for any $h \in \H$.
	\item[{\rm (ii)}] $\V_\M \leq \V(\Id)$.
	\end{enumerate}
\end{corollary}
\begin{proof}
(i) $\Rightarrow$ (ii) Let $\A = \langle \{1\}; \emptyset \rangle$ be a $1$-element algebra whose language is the empty set. Let $\A_\M = \langle \{0,1\}; \H \rangle$ be the constructed algebra of \cref{cube}. For any positive integer $m$ and tuples $a_1,\dots, a_n \in \{0,1\}^m$, by \cref{elim} we see that if $(1,\dots, 1) \in \langle a_1,\dots, a_n \rangle_{\A_\M^m}$, then there is a term $p$ in the language $\emptyset$ such that $p^{\A_\M^m}(a_1, \dots a_n) = (1,\dots,1)$; that is, $a_i = (1,\dots, 1)$ for some $1 \leq i \leq n$. Thus, the operations of $\A_\M$ preserve the relation $R_m = \{0,1\}^m \setminus (1,\dots,1)$ for all $m \geq 1$. 

Let $\mathbf{R}$ be the relational structure $\langle \{0,1\}; \{R_m\}_{m\geq 1} \rangle$, and let $\mathrm{Pol}(\mathbf{R})$ denote the set of all operations which preserve the relations of $\mathbf{R}$. Let $\mathrm{Clo}(\Id)$ denote the clone of term operations of $\Id$. Since the operations of $\A_\M$ are a subset of $\mathrm{Pol}(\mathbf{R})$ and $\mathrm{Pol}(\mathbf{R}) = \mathrm{Clo}(\Id)$ \cite{Post}, for every $h \in \H$ of arity $k$, there is a term $t_h$ of arity $k$ in the language of $\Id$ such that $h^{\A_\M}(c_1,\dots,c_k) = t_h^{\Id}(c_1,\dots,c_k)$ for all $(c_1,\dots, c_k) \in \{0,1\}^k$. Thus, $\langle \{0,1\}; \{t_h^{\Id}\}_{h \in \H}\rangle \models \Sigma$. If $\B  \in \V(\Id)$, then $\B$ satisfies all identities satisfied by $\Id$, so $\langle B; \{t_h^{\B}\}_{h \in \H} \rangle \models \Sigma$. This implies that $\langle B; \{t_h^{\B}\}_{h \in \H} \rangle  \in \V_\M$. Hence, $\V_\M \leq \V(\Id)$.

(ii) $\Rightarrow$ (i) If $\V_\M \leq \V(\Id)$, then for each operation symbol $h \in \H$ there is a term $t_h$ in the language of $\Id$ such that the algebra $\A = \langle \{0,1\}; \{t_h^{\Id} \}_{h \in \H} \rangle$ is a member of $\V_\M$. Thus $\M$ is consistent. Since $\Id$ does not have a cube term, $\A$ cannot have a cube term. Thus, $\Sigma$ does not entail the existence of cube identities for any $h \in \H$.
\end{proof}

We may quasi-order strong linear Maltsev conditions by interpretability. That is, we say $\M_1 \leq \M_2$ if and only if $\V_{\M_1} \leq \V_{\M_2}$. By identifying varieties which interpret into each other, this becomes a partial order. If $\M_1 \leq \M_2$, we say $\M_2$ is \textit{stronger} than $\M_1$. 

Given a finite index set $J$ and finitely many strong linear Maltsev conditions indexed by $J$, $\M_j =(\H_j,\Sigma_j)$, we may form a new strong linear Maltsev condition $\M = (\bigcup_{j\in J} \H_j, \bigcup_{j\in J}\Sigma_j)$. 

\begin{lemma}\label{str}
If $\H_i \cap \H_j = \emptyset$ for all $i \neq j$, then the following are equivalent:
	\begin{enumerate}
	\item[{\rm (i)}] For all $j \in J$, $\M_j$ is consistent and $\Sigma_j$ does not entail the existence of cube identities for any $h \in \H_j$.
	\item[{\rm (ii)}] For all $j \in J$, $\V_{\M_j} \leq \V(\Id)$.
	\item[{\rm (iii)}] $\V_{\M} \leq \V(\Id)$.
	\item[{\rm (iv)}] $\M$ is consistent and $\bigcup_{j \in J}\Sigma_j$ does not entail the existence of cube identities for any $h \in\bigcup_{j\in J} \H_j$.
	\end{enumerate}
\end{lemma}
\begin{proof}
The equivalences (i) $\iff$ (ii) and (iii) $\iff$ (iv) follow from \cref{int}. We now show (ii) $\iff$ (iii).

If we assume, for all $j \in J$, there is a map from $\H_j$ to terms in the language of $\V(\Id)$, then we have an induced map from $\bigcup_{j \in J} \H_j$ to terms in the language of $\V(\Id)$. The induced map is well-defined since $\H_i \cap \H_j = \emptyset$ for all $i \neq j$. If we assume there is a map from $\bigcup_{j \in J} \H_j$ to terms in the language of $\V(\Id)$, then for all $j \in J$ we have an induced map from $\H_j$ to terms in the language of $\V(\Id)$ by restriction.

Let $\A \in \V(\Id)$. The algebra $\langle A; \{t_{h}^\A \}_{h \in \H_j} \rangle$ satisfies the identities in $\Sigma_j$ for all $j \in J$  if and only if the algebra $\langle A; \{t_h^\A\}_{h \in \bigcup_{j\in J} \H_j} \rangle$ satisfies the identities in $\bigcup_{j \in J} \Sigma_j$. Thus 
\[
\langle A; \{t_h^\A\}_{h \in \H_j} \rangle \in \V_{\M_j} \text{ for all $j \in J$} \iff \langle A; \{t_h^\A\}_{h \in \bigcup_{j\in J} \H_j} \rangle \in \V_\M,
\]
which shows $\V_{\M_j} \leq \V(\Id)$ for all $j \in J$ if and only if $\V_\M \leq \V(\Id)$. 
\end{proof}

Thus from finitely many strong linear Maltsev conditions for which \cref{cube} applies, we may produce a stronger strong linear Maltsev condition for which \cref{cube} applies. We will use this strategy to obtain examples of finite algebras in varieties that are congruence distributive and congruence $k$-permutable ($k \geq 3$) whose subpower membership problem is EXPTIME-complete.

B. J\'{o}nsson \cite{Jon} characterized algebras in congruence distributive varieties by the existence of an integer $k \geq 1$ and ternary terms $d_0, \dots, d_k$ which satisfy the following set of identities:
	\begin{align*}
	d_0(x,y,z) &\approx x, \\
	d_k(x,y,z) &\approx z ,\\
	d_i(x,y,x) &\approx x \text{ for all } 0 \leq i \leq k ,\\
	d_i(x,x,y) &\approx d_{i+1}(x,x,y) \text{ for all even } i , \text{ and }\\
	d_i(x,y,y) &\approx d_{i+1}(x,y,y) \text{ for all odd } i  .
	\end{align*} 
The terms $d_0, \dots, d_k$ are referred to as \textit{J\'{o}nsson terms}, and $\CD(k)$ is often used to refer to the class of algebras which have J\'{o}nsson terms $d_0, \dots, d_k$. 

J. Hagemann and A. Mitschke \cite{HM} characterized algebras in congruence $k$-permutable varieties by the existence of ternary terms $p_0, \dots, p_k$ which satisfy the following set of identities:
	\begin{align*}
	p_0(x,y,z) &\approx x, \\
	p_k(x,y,z) &\approx z, \text{ and } \\
	p_i(x,x,y ) &\approx p_{i+1}(x,y,y) \text{ for all } i .
	\end{align*}
The terms $p_0, \dots, p_k$ are referred to as \textit{Hagemann--Mitschke terms}, and $\CP(k)$ is often used to refer to the class of algebras which have Hagemann--Mitschke terms $p_0, \dots, p_k$. 

We note that the sequence of the conditions $\CD(k)$ (respectively, $\CP(k)$) is a weakening sequence; that is, if $\A$ is a member of $\CD(k)$ (respectively, $\CP(k)$), $\A$ is also a member of $\CD(\ell)$ (respectively, $\CP(\ell)$) for all $\ell > k$. 

An algebra is in $\CD(1)$ if and only if it is trivial, and is in $\CD(2)$ if and only if it has a majority term operation. If an algebra is in $\CP(2)$ and is also in a congruence distributive variety, then the algebra has a majority term operation \cite{Pix}. Thus, every finite algebra which satisfies one of these properties has a subpower membershp problem in P.

\begin{corollary}\label{cdcp}
If $k \geq 3$ and $\ell \geq 3$, then there exists a finite algebra $\A \in \CD(k) \cap \CP(\ell)$ such that $\SMP(\A)$ is {\rm EXPTIME}-complete.
\end{corollary}
\begin{proof}
Let $\M_1 = (\H_1,\Sigma_1)$ be the strong linear Maltsev condition for $\CD(3)$. Note that the boolean operation $\wedge$ is a term operation of $\Id$ given by $x \wedge y = (x \tod y) \tod y$. It is straightforward to check that the terms
	\begin{align*}
	d_1(x,y,z) &= ((y \tod x) \wedge (z \tod x)) \tod x, \\
	d_2(x,y,z) &= (x \tod y) \tod z,
	\end{align*}
and the projections $d_0$ and $d_3$ satisfy the identities of $\CD(3)$, and so $\V_{\M_1} \leq \V(\Id)$. By \cref{int}, $\M_1$ is consistent and $\Sigma_1$ does not entail the existence of cube identities for any $h \in \H_1$. 

Let $\M_2 = (\H_2, \Sigma_2)$ be the strong linear Maltsev condition for $\CP(3)$. It is straightforward to check that the terms
	\begin{align*}
	p_1(x,y,z) &= (z \tod y) \tod x, \\
	p_2(x,y,z) &= (x \tod y) \tod z,
	\end{align*}
and the projections $p_0$ and $p_3$ satisfy the identities of $\CP(3)$, and so $\V_{\M_2} \leq \V(\Id)$. By \cref{int}, $\M_2$ is consistent and $\Sigma_2$ does not entail the existence of cube identities for any $h \in \H_2$. 

By \cref{str}, $\M = (\H_1 \cup \H_2, \Sigma_1 \cup \Sigma_2)$ is consistent and $\Sigma_1 \cup \Sigma_2$ does not entail the existence of cube identities for any $h \in \H_1 \cup \H_2$. Then by \cref{cube}(ii), there exists $\A_\M \in \CD(3) \cap \CP(3)$ (thus in $\CD(k) \cap \CP(\ell)$ for $k,\ell \geq 3$) such that $\SMP(\A_\M)$ is EXPTIME-complete.
\end{proof}

There exist examples of finite semigroups whose SMP is $\mathrm{NP}$-complete and examples of finite semigroups whose SMP is $\mathrm{PSPACE}$-complete \cite{BKMS,St1,St2}. We know from \cref{cube}(i) that if we expand these semigroups to algebras that belong to congruence distributive or congruence $k$-permutable ($k \geq 3$) varieties, the subpower membership problem for the expanded algebra is at least as hard as the subpower membership problem for the original algebra. The upper bound for the complexity of these problems remains unknown. 

\begin{question}
Are there examples of algebras in congruence distributive varieties or congruence $k$-permutable varieties whose SMP is $\mathrm{NP}$-complete or $\mathrm{PSPACE}$-complete?
\end{question}

\begin{ack}
I would like to thank \'{A}gnes Szendrei for many helpful suggestions and enlightening discussions. I would also like to thank Peter Mayr for his helpful suggestions. 
\end{ack}


\end{document}